\documentclass[11pt,a4paper]{article}
\usepackage{amsmath}
\usepackage{amsthm}
\usepackage{amssymb}
\usepackage{amscd}
\usepackage{graphicx}
\usepackage{epsfig}
\usepackage[matrix,arrow,curve]{xy}

\usepackage{latexsym}
\usepackage{amsfonts}
\input xy
\usepackage{tikz}
\usetikzlibrary{arrows,chains,matrix,positioning,scopes,snakes}
\makeatletter
\tikzset{join/.code=\tikzset{after node path={%
\ifx\tikzchainprevious\pgfutil@empty\else(\tikzchainprevious)%
edge[every join]#1(\tikzchaincurrent)\fi}}}
\makeatother

\tikzset{>=stealth',every on chain/.append style={join},
         every join/.style={->}}
\tikzset{
    >=stealth',
    punkt/.style={
           rectangle,
           rounded corners,
           draw=black, very thick,
           text width=6.5em,
           minimum height=2em,
           text centered},
    pil/.style={
           ->,
           thick,
           shorten <=2pt,
           shorten >=2pt,}
}

\xyoption{all} \tolerance=500

\newtheorem{thm}{Theorem}[section]
\newtheorem{proposition}[thm]{Proposition}

\theoremstyle{definition}
\newtheorem{example}[thm]{Example}
\newtheorem{remark}[thm]{Remark}
\newtheorem{definition}[thm]{Definition}

\font\black=cmbx10 \font\sblack=cmbx7 \font\ssblack=cmbx5 \font\blackital=cmmib10  \skewchar\blackital='177
\font\sblackital=cmmib7 \skewchar\sblackital='177 \font\ssblackital=cmmib5 \skewchar\ssblackital='177
\font\sanss=cmss12 \font\ssanss=cmss8 scaled 900 \font\sssanss=cmss8 scaled 600 \font\blackboard=msbm10
\font\sblackboard=msbm7 \font\ssblackboard=msbm5 \font\caligr=eusm10 \font\scaligr=eusm7 \font\sscaligr=eusm5

\font\bsymb=cmsy10 scaled\magstep2
\def\all#1{\setbox0=\hbox{\lower1.5pt\hbox{\bsymb
       \char"38}}\setbox1=\hbox{$_{#1}$} \box0\lower2pt\box1\;}
\def\exi#1{\setbox0=\hbox{\lower1.5pt\hbox{\bsymb \char"39}}
       \setbox1=\hbox{$_{#1}$} \box0\lower2pt\box1\;}

\newfam\bifam
\textfont\bifam=\blackital \scriptfont\bifam=\sblackital \scriptscriptfont\bifam=\ssblackital

\newfam\blfam
\textfont\blfam=\black \scriptfont\blfam=\sblack \scriptscriptfont\blfam=\ssblack

\newfam\bbfam
\textfont\bbfam=\blackboard \scriptfont\bbfam=\sblackboard \scriptscriptfont\bbfam=\ssblackboard

\newfam\ssfam
\textfont\ssfam=\sanss \scriptfont\ssfam=\ssanss \scriptscriptfont\ssfam=\sssanss
\def\sss#1{{\fam\ssfam\relax#1}}

\newfam\clfam
\textfont\clfam=\caligr \scriptfont\clfam=\scaligr \scriptscriptfont\clfam=\sscaligr

\def\pmb#1{\setbox0\hbox{${#1}$} \copy0 \kern-\wd0 \kern.2pt \box0}
\def\pmbb#1{\setbox0\hbox{${#1}$} \copy0 \kern-\wd0
      \kern.2pt \copy0 \kern-\wd0 \kern.2pt \box0}
\def\pmbbb#1{\setbox0\hbox{${#1}$} \copy0 \kern-\wd0
      \kern.2pt \copy0 \kern-\wd0 \kern.2pt
    \copy0 \kern-\wd0 \kern.2pt \box0}
\def\pmxb#1{\setbox0\hbox{${#1}$} \copy0 \kern-\wd0
      \kern.2pt \copy0 \kern-\wd0 \kern.2pt
      \copy0 \kern-\wd0 \kern.2pt \copy0 \kern-\wd0 \kern.2pt \box0}
\def\pmxbb#1{\setbox0\hbox{${#1}$} \copy0 \kern-\wd0 \kern.2pt
      \copy0 \kern-\wd0 \kern.2pt
      \copy0 \kern-\wd0 \kern.2pt \copy0 \kern-\wd0 \kern.2pt
      \copy0 \kern-\wd0 \kern.2pt \box0}

\mathchardef\za="710B  
\mathchardef\zb="710C  
\mathchardef\zg="710D  
\mathchardef\zd="710E  
\mathchardef\zve="710F 
\mathchardef\zz="7110  
\mathchardef\zh="7111  
\mathchardef\zvy="7112 
\mathchardef\zi="7113  
\mathchardef\zk="7114  
\mathchardef\zl="7115  
\mathchardef\zm="7116  
\mathchardef\zn="7117  
\mathchardef\zx="7118  
\mathchardef\zp="7119  
\mathchardef\zr="711A  
\mathchardef\zs="711B  
\mathchardef\zt="711C  
\mathchardef\zu="711D  
\mathchardef\zvf="711E 
\mathchardef\zq="711F  
\mathchardef\zc="7120  
\mathchardef\zw="7121  
\mathchardef\ze="7122  
\mathchardef\zy="7123  
\mathchardef\zf="7124  
\mathchardef\zvr="7125 
\mathchardef\zvs="7126 
\mathchardef\zf="7127  
\mathchardef\zG="7000  
\mathchardef\zD="7001  
\mathchardef\zY="7002  
\mathchardef\zL="7003  
\mathchardef\zX="7004  
\mathchardef\zP="7005  
\mathchardef\zS="7006  
\mathchardef\zU="7007  
\mathchardef\zF="7008  
\mathchardef\zW="700A  

\newcommand{\be}{\begin{equation}}
\newcommand{\ee}{\end{equation}}

\newcommand{\bea}{\begin{eqnarray}}
\newcommand{\eea}{\end{eqnarray}}
\newcommand{\beas}{\begin{eqnarray*}}
\newcommand{\eeas}{\end{eqnarray*}}
\def\*{{\textstyle *}}

\def\Sec{\sss{Sec}}

\newdir{|>}{%
!/4.5pt/@{|}*:(1,-.2)@^{>}*:(1,+.2)@_{>}}

\newcommand{\Ci}{C^{\infty}}

\begin{document}
\title{\bf Remarks on generalized Lie algebroids\\ and related concepts\thanks{Research funded by the  Polish National Science Centre grant under the contract number DEC-2012/06/A/ST1/00256.  }}
\date{Jan. 25, 2017}
\author{\\ KATARZYNA  GRABOWSKA$^1$\\ JANUSZ GRABOWSKI$^2$\\
        \\
         $^1$ {\it Faculty of Physics}\\
                {\it University of Warsaw}\\
                \\$^2$ {\it Institute of Mathematics}\\
                {\it Polish Academy of Sciences}
                }

\maketitle

\begin{abstract}{We prove that ``generalized Lie algebroid", a geometric object which appeared recently in the literature, is a misconception.}
\end{abstract}

\vspace{2mm} \noindent {\bf MSC 2010}: (primary) 17B66, 53D17;
 (secondary) 22A22, 53C15, 70G45, 70H33, 70S05.


\section{Introduction}
\emph{Lie algebroids} are objects which appear almost everywhere in contemporary geometry. They form a common generalization of a Lie algebra and tangent bundle $T M$, and are so common and natural that we are often working with them without mentioning it. Being related to almost all areas of geometry, like connection theory, cohomology theory, invariants of foliations and pseudogroups, symplectic and Poisson geometry, etc., and having a huge impact on mathematical physics, Lie algebroids became an object of extensive studies. As it is written in \cite{Gr1}, people told that they are using a Lie algebroid resemble Molier's Mr. Jourdain who was surprised and delighted to learn that he has been speaking prose all his life without knowing it.

The concept of Lie algebroid was introduced in 1967 by Pradines \cite{Pr}, as the basic infinitesimal invariant of a differentiable groupoid; the construction of the Lie algebroid
of a differentiable groupoid is a direct generalization of the construction of
the Lie algebra of a Lie group, and \cite{Pr} described a full Lie theory for differentiable
groupoids and Lie algebroids. In the same year the booklet \cite{Ne} by Nelson was published, where a general theory of \emph{Lie modules} together with a big part of the
corresponding differential calculus can be found.

A \emph{Lie pseudoalgebra},  a pure algebraic   counterpart of a Lie algebroid,
appeared first in the paper of Herz \cite{He} but one can find similar concepts under more than a dozen  of names  in the   literature (e.g., \emph{Lie module,  $(R,A)$-Lie
algebras, Lie-Cartan pair, Lie-Rinehart algebra, differential
algebra}, etc.).  We refer to a survey article by Mackenzie \cite{Mac}, his book \cite{Mac1}, the survey \cite{Gr1}, and references therein.

A Lie algebroid can be defined as a linear Poisson structure, a sort of de Rham derivative, a homological vector field on a supermanifold, a morphism of certain double vector bundles, etc., but the traditional definition is the following.

\begin{definition} A \emph{Lie algebroid} $(A, [\cdot,\cdot ]_A, \zr_A)$ is a smooth real
vector bundle $A$ over base $M$, with a real Lie algebra structure $[\cdot ,\cdot ]_A$ on the vector space $\Sec(A)$ of the smooth global sections of $A$, and a morphism of vector bundles  $\zr_A : A\to T M$ covering the identity on $M$  (called the anchor), such that
\be\label{LA}[X, fY ]_A = f[X, Y ]_A + \zr_A(X)(f)Y\,,
\ee
for all $X, Y \in\Sec(A)$ and $f\in \Ci(M)$.
\end{definition}
Sometimes, the axiom telling that $\zr_A$ induces a homomorphism of the Lie algebroid bracket and the bracket of vector fields,
$$\zr_A([X,Y]_A)=[\zr_A(X),\zr_A(Y)]\,,
$$
is added, but this axiom is a simple consequence of the above definition (cf. \cite[Theorem 2]{Gr}).

The classical Cartan differential calculus on a manifold  $M$, which can be viewed as associated with the canonical Lie algebroid structure on the tangent bundle $T M$ represented by the Lie bracket of vector fields, can be extended to calculus on a general Lie algebroid (cf. \cite{Mar}).
This includes the cohomology theory and the concept of characteristic classes which is of purely Lie algebroid nature, as observed by Kubarski \cite{Ku}.

The idea of developing Analytical Mechanics on Lie groupoids and algebroids, initiated by Weinstein \cite{We} and Libermann \cite{Li}, was elaborated by many authors, in the framework of Lie algebroid prolongations \cite{CaM,CMMD,dLMM,Mar}, as well as in the framework of Tulczyjew triples \cite{Gra06,Gra08,Gra11}. This concerns also the Lie algebroid version of the theory of optimal control \cite{CoM,GJ,Mar1,Po}.

Also the Lie theory for Lie groupoids and algebroids has been far developed, including the context of additional geometric structures  in the spirit of Drinfel'd \cite{Dri} (e.g. symplectic/Poisson Lie groupoids and Lie-bialgebroids) and the solution of the problem of integration \cite{CF}.

\medskip
It is not at all surprising that there appeared various generalizations of the concept of Lie algebroid, e.g. \emph{general}, \emph{skew}, and \emph{almost} Lie algebroids (see e.g. \cite{GJ,GU,GU1}, where the Jacobi identity is not assumed, or \emph{Loday} or \emph{Leibniz}
algebroids \cite{GKP,ILMP} where Jacobi is kept but not skew-symmetry. Applications of these structures to mechanics and optimal control have also been developed \cite{Gra06,Gra08,Gra11,GJ,GLMM}.

A whole branch of preprints (with substantial overlaps) devoted to certain generalizations of Lie algebroids and its applications, including theory of connections, mechanics and optimal control, was published by Arcu\c{s} and his collaborators \cite{A}--\cite{A11}. Roughly speaking, the generalization depends on relaxing the assumptions that the anchor map is a vector bundle morphism which is supported on the identity on the base manifold. Some of these papers appeared recently in respected journals of mathematical physics \cite{A8,A11}.

\medskip
The aim of this note is to show that these generalizations are illusory, as they actually reduce to the case of standard Lie algebroids. Hence, the whole theory of ``generalized Lie algebroids", together with the associated geometry and applications, is void and brings nothing to our understanding of geometry and physics. This concerns, in particular, papers \cite{A8,A11} and explains, why no natural (true) examples are offered.

\section{Generalized Lie algebras and algebroids}
In what follows, we will present the original definition of a \emph{generalized Lie algebroid}, as it is formulated in \cite{A11}. Since this definition refers to the concept of a \emph{Lie $\mathcal{F}(N)$-algebra}, which is not explained in \cite{A11}, we start with the latter following \cite{A10}. We cite the layout and definitions \emph{in extenso}, keeping the original style and notation of \cite{A10,A11}, to give the reader the flavour of the original presentation. Note only that in what follows we understand $\mathcal{F}$ as an associative commutative ring with a unit, and the triple $(F,\nu,N)$ as the vector bundle $\nu:F\rightarrow N$ (this is not specified in the cited papers). Our own comments are postponed to the next section.

\subsection{Generalized Lie algebras}
\begin{definition}\label{d1}
If $A$ is a $\mathcal{F}$-module such that there exists an
biaditive operation
\begin{equation*}
\begin{array}{ccc}
A\times A & ^{\underrightarrow{~\ \ [\ ,\ ] _{A}~\ \ }} & A \\
( u,v) & \longmapsto & [ u,v] _{A}%
\end{array}
,
\end{equation*}%
then we say that $( A, [ ,] _{A}) $ is a $\mathcal{F}$-algebra or algebra over $\mathcal{F}$.
\end{definition}
If $(A, [,]_A)$ is a $\mathcal{F}$-algebra such that the operation $%
[,]_A$ is associative (commutative), then  $(A, [,]_A)$ is called an
associative (commutative) $\mathcal{F}$-algebra. Moreover, $(A, [,]_A)$ is called a unitary $\mathcal{F}$-algebra, if the
operation $[,]_A$ has a unitary element.
\begin{remark}
In the above definition, if $\mathcal{F}$ is a commutative ring and $[,]_A$ is bilinear, then we have the classical definition
of algebra over a ring. Thus every classical $\mathcal{F}$-algebra is an $\mathcal{F}$-algebra, but the converse is not true. For example if $M$
is a manifold, then $(\chi(M), [,])$, where
\begin{equation*}
[X, Y](f)=X(Y(f))-Y(X(f)),\ \ \ \forall X, Y\in\chi(M),\ \ \forall f\in \mathcal{F}(M),
\end{equation*}
 is an $\mathcal{F}(M)$-algebra
but it is not a classical $\mathcal{F}(M)$-algebra (only a classical $%
\mathbb{R}$-algebra, where $\mathbb{R}$ is the field of real
numbers).
\end{remark}
\begin{definition}
If $\mathcal{F}$ is a ring, then the set $Der(
\mathcal{F}) $\ of groups morphisms $X:\mathcal{F\longrightarrow F}$
satisfying the condition 
\begin{equation*}
X( f\cdot g) =X( f) \cdot g+f\cdot X( g) ,~\forall f,g\in \mathcal{F},
\end{equation*}
will be called the set of derivations of  $\mathcal{F}$.
\end{definition}
If $M$ is a manifold, then it is easy to check that $Der( \mathcal{F}( M) )=\chi(M)$.
\begin{example}
If we consider the
biadditive operation
\begin{equation*}
\begin{array}{ccc}
Der( \mathcal{F}) \times Der( \mathcal{F}) & ^{\underrightarrow{~\ \ [\
,\ ]_{Der(\mathcal{F})} ~\ \ }} & Der( \mathcal{F}) \\
( X,Y) & \longmapsto & [ X,Y]_{Der(\mathcal{F})}%
\end{array}
,
\end{equation*}%
given by
\begin{equation*}
[ X,Y]_{Der(\mathcal{F})}( f) =X( Y( f) ) -Y( X( f) ) ,~\forall f\in
\mathcal{F},
\end{equation*}
then $( Der( \mathcal{F})  ,[ ,]_{Der(\mathcal{F})} ) $ is an $\mathcal{F}$-algebra.
\end{example}
\begin{definition}\label{Def12}
If $( A,[ ,] _{A}) $ is an $\mathcal{F}$-algebra such that $[ ,]
_{A} $ satisfies the conditions:

$LA_{1}.$ $[ u,u] _{A}=0,$ for any $u\in A$,

$LA_{2}.$ $[ u,[ v,z] _{A}] _{A}+[ z,[ u,v] _{A}] _{A}+[ v,[ z,u] _{A}]
_{A}=0,$ for any $u,v,z\in A$,\newline
then we will say that $( A,[ ,] _{A}) $ is a Lie $\mathcal{F}$%
-algebra or Lie algebra over $\mathcal{F}$.
\end{definition}
\begin{example}
It is easy to check that the $\mathcal{F}$-algebra $( Der( \mathcal{F})
,[,]_{Der(\mathcal{F})}) $ is a Lie $\mathcal{F}$-algebra.
\end{example}
Using Definition \ref{Def12} we deduce the following:
\begin{proposition}
If $( A,[ ,] _{A}) $ is a Lie $\mathcal{F}$-algebra, then we have:

1. $[ u,v] _{A}=-[ v,u] _{A},$ for any $u,v\in A$,

2. $[ u,0] _{A}=0,$ for any $u\in A$,

3. $[ -u,v] _{A}=-[ u,v] _{A}=[ u,-v] _{A}$, for any $u,v\in A.$
\end{proposition}

\begin{definition}
Let $ A $ be an $\mathcal{F}$-module. If there exists a modules
morphism $\rho $ from $ A$ to $Der( \mathcal{F}) $
and a biadditive operation
\begin{equation*}
\begin{array}{ccc}
A\times A & ^{\underrightarrow{~\ \ [\ ,\ ] _{A}~\ \ }} & A \\
( u,v) & \longmapsto & [ u,v] _{A}%
\end{array}
,
\end{equation*}%
satisfies in
\begin{equation}
[ u,fv] _{A}=f[ u,v] _{A}+\rho ( u) ( f) \cdot v,
\end{equation}%
for any $u,v\in A$ and $f\in \mathcal{F}$ such that $( A,[ ,]
_{A}~) $ is a Lie $\mathcal{F}$-algebra, then $( A ,[
,] _{A}, \rho ) $ is called a generalized Lie $\mathcal{F}$-algebra or
generalized Lie algebra over $ \mathcal{F}$.
\end{definition}

\subsection{Generalized Lie algebroids}
Let $(F,\nu ,N)$ be a vector bundle, $\Gamma ( F, \nu, N)$ be the set of the sections of it and $\mathcal{F}(
N)$ be the set of smooth real-valued functions on $N$. Then $(\Gamma
(F,\nu ,N) ,+,\cdot)$ is a $\mathcal{F}(N)$-module. If $(\varphi ,\varphi _{0})$ is a morphism from $( F,\nu ,N)$ to $(F', \nu', N')$ such that $\varphi _{0}$ is a isomorphism from $N$ to $N'$, then using the operation
\begin{equation*}
\begin{array}{ccc}
\mathcal{F}\left( N\right) \times \Gamma \left(F', \nu', N'\right) & ^{\underrightarrow{~\ \ \cdot ~\ \ }} & \Gamma \left(
F', \nu', N'\right),\\
\left( f,u^{\prime }\right) & \longmapsto & f\circ \varphi _{0}^{-1}\cdot
u^{\prime },
\end{array}%
\end{equation*}%
it results that $(\Gamma( F', \nu', N') ,+,\cdot)$ is a $\mathcal{F}(N)$-module and the modules morphism
\begin{equation*}
\begin{array}{ccc}
\Gamma \left(F,\nu ,N\right) & ^{\underrightarrow{~\ \ \Gamma \left(
\varphi ,\varphi _{0}\right) ~\ \ }} & \Gamma \left( F', \nu', N'\right),\\
u & \longmapsto & \Gamma \left( \varphi ,\varphi _{0}\right) u,
\end{array}%
\end{equation*}%
defined by
\begin{equation*}
\begin{array}{c}
\Gamma ( \varphi ,\varphi _{0}) u(y) =\varphi ( u_{\varphi
_{0}^{-1}(y)}) =( \varphi \circ u\circ
\varphi _{0}^{-1})(y),%
\end{array}%
\end{equation*}%
for any $y\in N'$ will be obtained.
\begin{definition}\label{AP}
A generalized Lie algebroid is a vector bundle $(F,\nu ,N)$ given by the diagram
\begin{equation}
\begin{array}{c}
\begin{array}[b]{ccccc}
\left( F,\left[ ,\right] _{F,h}\right) & ^{\underrightarrow{~\ \ \ \rho \ \
\ \ }} & \left( TM,\left[ ,\right] _{TM}\right) & ^{\underrightarrow{~\ \ \
Th\ \ \ \ }} & \left( TN,\left[ ,\right] _{TN}\right)\\
~\downarrow \nu &  & ~\ \ \downarrow \tau _{M} &  & ~\ \ \downarrow \tau _{N}
\\
N & ^{\underrightarrow{~\ \ \ \eta ~\ \ }} & M & ^{\underrightarrow{~\ \ \
h~\ \ }} & N
\end{array}
\end{array}
\end{equation}
where $h$ and $\eta $ are arbitrary isomorphisms, $(\rho, \eta): (F,\nu, N)\longrightarrow (TM,\tau _{M},M)$ is a vector bundles morphism and
\begin{equation*}
\begin{array}{ccc}
\Gamma \left( F,\nu ,N\right) \times \Gamma \left( F,\nu ,N\right) & ^{%
\underrightarrow{~\ \ \left[ ,\right] _{F,h}~\ \ }} & \Gamma \left( F,\nu
,N\right),\\
\left( u,v\right) & \longmapsto & \ \left[ u,v\right] _{F,h},
\end{array}
\end{equation*}
is an operation satisfies in
\begin{equation}\label{ganchor}
\begin{array}{c}
\left[ u,f\cdot v\right] _{F,h}=f\left[ u,v\right] _{F,h}+\Gamma \left(
Th\circ \rho ,h\circ \eta \right) \left( u\right) f\cdot v,\ \ \ \forall f\in \mathcal{F}(N),
\end{array}
\end{equation}
such that the 4-tuple $(\Gamma( F, \nu, N) ,+,\cdot, [ , ] _{F,h})$ is a Lie $\mathcal{F}(N)$-algebra.
\end{definition}
We denote by $\Big((F, \nu, N), [ , ] _{F,h}, (\rho, \eta) \Big)$ the generalized Lie algebroid defined in the above. Moreover, the couple $\Big([ , ]
_{F,h}, (\rho, \eta)\Big)$ is called the \emph{generalized
Lie algebroid structure}. It is easy to see that in the above definition, the modules morphism $\Gamma(Th\circ\rho, h\circ \eta) : (\Gamma(F,\nu ,N), +, \cdot, [ , ] _{F,h})\longrightarrow(\Gamma(TN, \tau_{N}, N), +, \cdot, [ , ]_{TN})$ is a Lie algebras morphism.
\begin{definition}
A morphism from
$
(( F,\nu ,N), [ , ] _{F,h}, (\rho, \eta))
$
to
$
(( F^{\prime }, \nu ^{\prime }, N^{\prime }), [ , ]
_{F^{\prime }, h^{\prime }}, (\rho ^{\prime }, \eta ^{\prime }))
$
is a morphism $(\varphi ,\varphi _{0})$ from $(F,\nu, N)$ to $(F',\nu', N')$ such that $\varphi _{0}$ is an isomorphism from $N$ to $N'$, and the modules morphism $\Gamma(\varphi, \varphi _{0})$
is a Lie algebras morphism from
$
( \Gamma \left( F,\nu ,N\right) ,+,\cdot ,\left[ ,\right] _{F,h})
$
to
$
( \Gamma \left( F^{\prime },\nu ^{\prime },N^{\prime }\right) ,+,\cdot
, \left[ ,\right] _{F^{\prime },h^{\prime }}).
$
\end{definition}
\begin{remark}
In particular case $\left( \eta ,h\right) =\left(
Id_{M},Id_{M}\right)$, the definition of the Lie algebroid will be obtained.
\end{remark}

\section{Why it does not work}
Let us start with the remark that Definition \ref{d1} does not make much sense, since
it just collects two unrelated objects on $A$: a bi-additive operation $[, ]_A$ and an
$\mathcal{F}$-module structure. Usually in mathematics one considers some compatibility conditions between various structures, so the traditional definition of an $\mathcal{F}$-algebra seems to be better.
Sharing the authors opinion, we could as well ``generalize'' the concept of a topological vector space assuming that we deal with structures of a topological space and a vector space which are completely unrelated.

\medskip
Second, the authors ``generalized Lie algebra" is actually a particular case of a Lie pseudoalgebra (Lie-Rinehart algebra) mentioned in the introduction.

\medskip
Trying to guess what and why the authors require in Definition \ref{AP}, let us just mention that
what is really used to express the algebraic properties of the bracket $[, ]_{F,h}$ is only the composition $\Phi=Th\circ\rho$, covering $\phi=h\circ\eta$ and playing the role of the anchor map, so that working with two diffeomorphic base manifold is completely unmotivated. Moreover, any vector bundle morphism $\Phi:F\to TN$, covering a certain $\phi: N\to N$, can always be written in the form $\Phi=Th\circ\rho$ for a given diffeomorphism $h:M\to N$ and a vector bundle morphism $\zr:F\to TM$\,; just take $\rho=(Th)^{-1}\circ\Phi$.

The condition (\ref{ganchor}) can then be rewritten as
\begin{equation}\label{ganchor1}
\left[ u,f\cdot v\right] _\Phi=f\left[ u,v\right] _\Phi+\tilde\Phi\left( u\right) f\cdot v,\ \ \ \forall f\in \mathcal{F}(N)\,,
\end{equation}
where we logically write $[, ]_\Phi$ instead of $[,]_{F,h}$ and $\widetilde{\Phi}$ for the map
$\Gamma \left( Th\circ \rho ,h\circ \eta \right) $ induced by $\Phi$ on sections.

Hence, the only relevant difference with the standard definition of a Lie algebroid is that the anchor map covers a diffeomorphisms which can differ from the identity. This is, however, an illusory generalization, since in (\ref{ganchor1}) we can replace $\Phi$ with a vector bundle morphism $\Psi:F\to TN$, this time covering the identity, without changing the bracket. It is possible, because  the diffeomorphism $\phi$ must be identical on the support of the ``anchor map" $\Phi$. Outside the support of $\Phi$ we can always change base points for free, but in geometrically interesting cases, i.e. when the anchor does not vanish on an open and dense subset of $N$, it must cover the identity. We can formulate this observation as follows.
\begin{thm}\label{t-main}
If \,$[, ]_\Phi$ is a Lie bracket on sections of a vector bundle $\zn:F\to N$ such that (\ref{ganchor1}) is satisfied for a vector bundle morphism $\Phi:F\to TN$ covering a diffeomorphism $\phi:N\to N$, then $[, ]_\Phi$ is a Lie algebroid bracket on $F$ with respect to a ``traditional" anchor map $\Psi:F\to TN$ covering the identity, $\Psi=(T\phi)^{-1}\circ\Phi$.
\end{thm}
\begin{proof}
The existence of an anchor, i.e. (\ref{ganchor1}), implies that $[, ]_\Phi$ is \emph{local} with respect to the second argument, i.e. $\left[ u,f\cdot v\right] _\Phi$ vanishes at $x\in N$ if $f$ vanishes in a neighbourhood of $x$. As the bracket is skew-symmetric, it is local also with respect to the first argument, so $[g\cdot u,f\cdot v]_\Phi$ must vanish at $x\in N$ if $g$ vanishes in a neighbourhood of $x$, that implies that the anchor acting on sections is a local operator, so covers a diffeomorphism which is a local operator, thus identity.

More precisely, it suffices to prove that $\Phi$ and $\Psi$ induce the same maps on sections, $\widetilde{\Psi}=\widetilde{\Phi}$.
For, we will prove that $\Phi$
is zero as a linear map
\begin{equation}\label{e1}\Phi_{\phi^{-1}(x)}:F_{\phi^{-1}(x)}\to T_x N
\end{equation}
for points $x$ of $\operatorname{supp}(\Phi)$ (being the closure of the set $\{ y\in N\,|\, \phi(y)\ne y\}$), so
$$\Psi_{\phi^{-1}(x)}=(T\phi)^{-1}_x\circ\Phi_{\phi^{-1}(x)}$$
is zero as well.

Assume that $\phi(x)\ne x$. Then, one can find a smooth function $g$ on $N$ such that
$g$ vanishes in a neighbourhood of $x$ and $g\circ\phi^{-1}(x)=1$.
Using (\ref{ganchor1}) and skew-symmetry, we get
\begin{eqnarray}\label{e2}
[g\cdot u,f\cdot v]_\Phi&=&f[g\cdot u,v]_\Phi(x)+(g\circ \phi^{-1})\cdot\widetilde{\Phi}(u)f\cdot v\\
&=&fg[u,v]_\Phi-\widetilde{\Phi}(f\cdot v)g\cdot u+(g\circ \phi^{-1})\cdot\widetilde{\Phi}(u)f\cdot v\,.\nonumber
\end{eqnarray}
Here we used
$$\widetilde{\Phi}(g\cdot u)=(g\circ \phi^{-1})\widetilde{\Phi}(u)\,.
$$
As $g$ vanishes in a neighbourhood of $x$, the left-hand side of (\ref{e2}) vanishes at $x$. Moreover, $(g\circ\phi^{-1})(x)=1$ by the choice of $g$, thus
$\widetilde{\Phi}(u)f\cdot v$
vanishes at $x$ for any $u,v$ and $f$.
We conclude that the vector field $\widetilde{\Phi}(u)$ vanishes at $x$ for any $u$, which clearly implies that the linear map (\ref{e1}) is $0$.
\end{proof}

\begin{remark}
Of course, the possibility of changing the anchor by a diffeomorphism supported on its kernel is completely irrelevant from the point of view of the bracket, so in Definition \ref{AP} we deal with a genuine Lie algebroid over $N$ with a presence of the diffeomorphism $h$. However, this sort of a generalization is in our opinion useless.
\end{remark}

{
\begin{remark} One can try to redeem the whole concept assuming that the module structure $(f,v)\mapsto f\cdot v$ in (\ref{ganchor1}) is replaced by $(f,v)\mapsto f\odot v=(f\circ\phi^{-1})\cdot v$, but then we again deal with the standard Lie algebroid on ${F}$, this time with respect to the vector bundle projection $\phi\circ\zn:F\to N$.
\end{remark}
}

\begin{remark}
It can be shown (cf. \cite{Gr}) that any Lie bracket on sections of a vector bundle $F$ such that (this  covers condition (\ref{ganchor1}))
$$[u,f\cdot v]_\Phi=f[u,v]_\Phi+A(u,f)\cdot v\,,$$
where $A(u,f)$ is a function on $N$, is a Lie algebroid except for $F$ being of rank one. If $F$ is a rank 1 vector bundle, then we deal with a Jacobi bundle (Poisson or Jacobi bracket). The bracket admits an ``anchor'' which is a first-order differential operator and we deal with a rank 1 Lie algebroid if and only if this anchor is actually of order zero. However, in this picture all differential operators representing the anchor cover the identity.
\end{remark}


\small{\vskip1cm

\noindent Katarzyna GRABOWSKA\\
Faculty of Physics\\
                University of Warsaw} \\
               Pasteura 5, 02-093 Warszawa, Poland
                 \\Email: konieczn@fuw.edu.pl \\

\noindent Janusz GRABOWSKI\\ Polish Academy of Sciences\\ Institute of
Mathematics\\ \'Sniadeckich 8, 00-656 Warsaw,
Poland\\Email: jagrab@impan.pl \\

\end{document}